\documentclass[144]{article}
\usepackage{amsmath,amssymb,amsthm,color}

\newtheorem{theorem}{Theorem}[section]
\newtheorem{corollary}{Corollary}[section]

\newtheorem{claim}{Claim}[section]

\usepackage{color}

\def\M{{\mathfrak M}\, }
\def\G{{\mathfrak G}\, }
\def\R{{\mathfrak R}\, }

\begin{document}

\begin{center}{\bf \LARGE Additivity of $n$-Multiplicative Mappings of
Gamma Rings}\\
\vspace{.2in}
\noindent {\bf Aline Jaqueline de Oliveira Andrade}\\
{\it Federal University of ABC,\\
dos Estados Avenue, 5001,\\
09210-580, Santo Andr\'{e}, Brazil.}\\
e-mail: alinejaqueline.o.a@gmail.com\\
\vspace{.2in}
\noindent {\bf Gabriela Cotrim de Moraes}\\
{\it Federal University of ABC,\\
dos Estados Avenue, 5001,\\
09210-580, Santo Andr\'{e}, Brazil.}\\
e-mail: gabriela.moraes@ufabc.edu.br\\
\vspace{.2in}
\noindent {\bf Ruth Nascimento Ferreira}\\
{\it Federal University of Technology,\\
Professora Laura Pacheco Bastos Avenue, 800,\\
85053-510, Guarapuava, Brazil.}\\
e-mail: ruthnascimento@utfpr.edu.br\\
\vspace{.2in}
\noindent {\bf Bruno Leonardo Macedo Ferreira}\\
{\it Federal University of Technology,\\
Professora Laura Pacheco Bastos Avenue, 800,\\
85053-510, Guarapuava, Brazil.}\\
e-mail: brunolmfalg@gmail.com\\ 
e-mail: brunoferreira@utfpr.edu.br
\end{center}

\begin{abstract}
In this paper, we address the additivity of $n$-multiplicative isomorphisms and $n$-multiplicative derivations on Gamma rings. We proved that, if  $\M$ is a $\Gamma$-ring satisfying the some conditions, then any $n$-multiplicative isomorphism $\left(\varphi, \phi\right)$ of $\M$ onto an arbitrary
gamma ring is additive.
\end{abstract}
{\bf 2010 Mathematics Subject Classification.} 16Y99, 16N60. 
\\
\textbf{Keywords:} $n$-multiplicative isomorphisms, $n$-multiplicative derivations, Additivity, Gamma rings.

\newpage

\section{Introduction and Preliminaries}

A ring of endomorphism of a module plays a very important role in many parts
of mathematics, the property of ring itself is also clarified when we consider it as a ring of endomorphisms of a module. As a generalization of this idea, we can consider a set of homomorphisms of a module to another module which is closed under the addition and subtraction defined naturally but has no more a structure of a ring since we cannot define the product. However, suppose that we have an additive group $M$ consisting of homomorphisms of a module $A$ to a module $B$ and that we have also an additive group $N$ consisting of homomorphisms of $B$ to $A$. In this case we can define the product of three elements $f_1, g$ and $f_2$ where $f_1$ and $f_2$ are elements of $M$ and $g$
is an element of $N$. If this product $f_1gf_2$ is also an element of $M$ for every $f_1,g$ and
$f_2$, we say that $M$ is closed under the multiplication using $N$ between. Similarly, we can define that $N$ is closed under the multiplication using $M$ between. In this case let be $f_1, f_2$
and $f_3$ in $M$ and $g_1$ and $g_2$ in $N$ then we have
$$(f_1 g_1 f_2)g_2f_3 = f_1 g_1 (f_2g_2f_3) = f_1 (g_1 f_2g_2)f_3.$$
Motivated by this observation Nobusawa \cite{nobu} considered this situation abstractly and introduced the notion of $\Gamma$-rings as follows:
Let $\M$ be an additive group whose elements are denoted by $x,y,z,\ldots $, and $\Gamma$ be another additive group whose elements are $\alpha, \beta, \gamma,\ldots $. Suppose that $x\gamma y$ is defined to be an element of $\M$ and that $\gamma x \beta $ is defined to be an element of $\Gamma$ for every $x, y \in \M$ and $\gamma, \beta \in \Gamma$. If the products satisfy
the following three conditions 
\begin{enumerate}
\item[\it(i)] $(x+y)\alpha z= x\alpha z + y\alpha z,~ x\alpha (y+z)= x\alpha y + x\alpha z,~ x(\alpha +\beta)y= x\alpha y + x\beta y,$
\item[\it (ii)] $(x\alpha y)\beta z= x\alpha (y \beta z) = x(\alpha y \beta) z,$
\item[\it(iii)] $x \gamma y =0$ implies $\gamma = 0$,
\end{enumerate}
for all $x,y,z\in \mathfrak{M}$ and $\alpha ,\beta, \gamma \in \Gamma$, then $\M$ is called a $\Gamma$-ring.
Barnes \cite{barnes} weakened
the above conditions and defined the $\Gamma$-rings as follows: 
Let $\mathfrak{M}$ and $\Gamma$ be two abelian groups. We call $\mathfrak{M}$ a {\it $\Gamma$-ring} if the following conditions are satisfied:
\begin{enumerate}
\item[\it (i)] $x\alpha y\in \mathfrak{M},$
\item[\it(ii)] $(x+y)\alpha z= x\alpha z + y\alpha z,~ x\alpha (y+z)= x\alpha y + x\alpha z,$
\item[\it (iii)] $x(\alpha +\beta)y= x\alpha y + x\beta y,$
\item[\it(iv)] $(x\alpha y)\beta z= x\alpha (y \beta z),$
\end{enumerate}
for all $x,y,z\in \mathfrak{M}$ and $\alpha ,\beta \in \Gamma.$
A nonzero element $1\in \mathfrak{M}$ is called a multiplicative $\gamma $-identity of $\mathfrak{M}$ or {\it $\gamma $-unity element} (for some $\gamma \in \Gamma$) if $1\gamma x=x\gamma 1=x$ for all $x\in \mathfrak{M}$. A nonzero element $e_{1}\in \mathfrak{M}$ is called a {\it $\gamma_{1} $-idempotent} (for some $\gamma _{1}\in \Gamma$) if $e_{1}\gamma _{1}e_{1}=e_{1}$ and a {\it nontrivial $\gamma_{1}$-idempotent} if it is a $\gamma_{1} $-idempotent different from multiplicative $\gamma_{1} $-identity element of $\mathfrak{M}$.

Let $\Gamma $ and $\mathfrak{M}$ be two abelian groups such that $\mathfrak{M}$ is a $\Gamma$-ring and $e_{1}\in\mathfrak{M}$ a nontrivial $\gamma_{1}$-idempotent. Let us consider $e_2 \colon \Gamma \times\mathfrak{M}\rightarrow\mathfrak{M}$ and $e'_2 \colon \mathfrak{M}\times \Gamma \rightarrow\mathfrak{M}$ two $\mathfrak{M}$-additive maps verifying the conditions $e_2 (\gamma_{1}, a)=a-e_{1}\gamma_{1}a$ and $e_2'(a,\gamma_{1})=a-a\gamma_{1}e_{1}$. Let us  denote $e_2\alpha a=e_2 (\alpha, a)$, $a \alpha e_2=e_2'(a,\alpha)$, $1_{1}\alpha a=e_1\alpha a+e_2\alpha a$, $a\alpha 1_{1}=a\alpha e_1+a\alpha e_2$ and suppose $(a\alpha e_{2})\beta b=a \alpha (e_{2}\beta b)$ for all $a,b\in \mathfrak{M}$ and $\alpha, \beta \in \Gamma$. Then $1_{1}\gamma_{1} a=a\gamma_{1} 1_{1}=a$ and $(a\alpha 1_{1})\beta b=a \alpha (1_{1}\beta b)$, for all $a,b\in \mathfrak{M}$ and $\alpha, \beta \in \Gamma$, allowing us to write $1_{1}=e_{1}+e_{2}$ and $\mathfrak{M}$ as a direct sum of subgroups $\mathfrak{M}=\mathfrak{M}_{11}\oplus \mathfrak{M}_{12}\oplus \mathfrak{M}_{21}\oplus \mathfrak{M}_{22},$ where $\mathfrak{M}_{ij}=e_{i}\gamma _{1}\mathfrak{M}\gamma _{1}e_{j}$ $(i,j=1,2),$ called {\it Peirce decomposition} of $\mathfrak{M}$ relative to $e_{1}$, satisfying the
multiplicative relations:
\begin{enumerate}
\item[\it (i)] $\mathfrak{M}_{ij}\Gamma \mathfrak{M}_{kl}\subseteq\mathfrak{M}_{il}$ $(i,j,k,l=1,2)$;
\item[\it (ii)] $\mathfrak{M}_{ij}\gamma_{1} \mathfrak{M}_{kl}=0$ if $j\neq k$ $(i,j,k,l=1,2)$.
\end{enumerate}

If $\mathfrak{A}$ and $\mathfrak{B}$ are subsets of a $\Gamma$-ring $\mathfrak{M}$ and $\Theta \subseteq \Gamma,$ we denote $\mathfrak{A}\Theta \mathfrak{B}$ the subset of $\mathfrak{M}$ consisting of all finite sums of the form $\sum_{i} a_{i}\gamma _{i}b_{i}$ where $a_{i}\in \mathfrak{A},$ $\gamma _{i}\in \Theta$ and $b_{i}\in \mathfrak{B}$. A {\it right ideal (resp., left ideal)} of a $\Gamma$-ring $\mathfrak{M}$ is an additive subgroup $\mathfrak{I}$ of $\mathfrak{M}$ such that $\mathfrak{I}\Gamma \mathfrak{M}\subseteq \mathfrak{I}$ (resp., $\mathfrak{M}\Gamma \mathfrak{I}\subseteq \mathfrak{I}$). If $\mathfrak{I}$ is both a right and a left ideal of $\mathfrak{M}$, then we say that $\mathfrak{I}$ is an {\it ideal} or {\it two-side ideal} of $\mathfrak{M}$.

An ideal $\mathfrak{P}$ of a $\Gamma$-ring $\mathfrak{M}$ is called {\it prime} if for any ideals $\mathfrak{A},\mathfrak{B}\subseteq \mathfrak{M}$, $\mathfrak{A}\Gamma \mathfrak{B}\subseteq \mathfrak{P}$ implies that $\mathfrak{A}\subseteq \mathfrak{P}$ or $\mathfrak{B}\subseteq \mathfrak{P}$. A $\Gamma$-ring $\mathfrak{M}$ is said to be {\it prime} if the zero ideal is prime.

\begin{theorem} If $\mathfrak{M}$ is a $\Gamma$-ring, the following conditions are equivalent:
\begin{enumerate}
\item[{\it (i)}] $\mathfrak{M}$ is a prime $\Gamma$-ring;
\item[{\it (ii)}] if $a,b\in \mathfrak{M}$ and $a\Gamma\mathfrak{M}\Gamma b = 0$, then $a=0$ or $b=0$.
\end{enumerate}
\end{theorem}

\section{$n$-multiplicative isomorphisms and derivations}

The first result about the additivity of maps on rings was given by Martindale
III \cite{mart}. He established a condition on a ring such that every multiplicative isomorphism on a ring $\R$ is additive. Wang \cite{wang} also considered this question in the context
of $n$-multiplicative maps on associative rings satisfying Martindale’s conditions. Ferreira and Ferreira \cite{ferfer} considered that problem in the context of alternative rings. They investigated the problem of when a $n$-multiplicative isomorphism and a $n$-multiplicative derivation is additive for the class of alternative rings satisfying Martindale’s conditions.

Let $n > 1$ be a positive integer. Let $\M$, $\G$, $\Gamma$ and $\Delta$ abelian groups such that $\M$ is a $\Gamma$-ring and $\G$ is a $\Delta$-ring (not necessarily with identity elements). An ordered pair ($\phi$, $\varphi$) of mappings is called a $n$-multiplicative isomorphism of $\M$ onto $\G$ if they satisfy the following properties:
\begin{enumerate}
\item [(i)] $\phi$ is a bijective mapping from $\M$ onto $\G$;
\item [(ii)] $\varphi$ is a bijective mapping from $\Gamma$ onto $\Delta$;
\item [(iii)] $\phi(x_1\gamma_1x_2\gamma_2\cdots\gamma_{n-1}x_n) = \phi(x_1)\varphi(\gamma_1)\phi(x_2)\varphi(\gamma_2)\cdots \varphi(\gamma_{n-1})\phi(x_n)$, for all
$x_1, x_2, \ldots , x_n \in \M$ and $\gamma_1, \gamma_2, \ldots , \gamma_{n-1} \in \Gamma$. 
\end{enumerate}
We say that a $n$-multiplicative isomorphism of $\M$ onto $\G$ ($\phi$, $\varphi$) is additive when $\phi(x + y) = \phi(x) + \phi(y)$ for all $x, y \in \M$. In particular, if $n = 2$ ($\phi$, $\varphi$) is called a multiplicative isomorphism. We say that a map $d\colon M \rightarrow \M$ is a $n$-multiplicative derivation of $\M$ if 
$$d(x_1\gamma_1x_2\gamma_2 \cdots\gamma_{n-1}x_n) = \sum_{i=1}^{n} x_1\gamma_1\cdots\gamma_{i-1}d(x_i)\gamma_i\cdots\gamma_{n-1}x_n$$
for all $x_1, x_2, \ldots , x_n \in \M$ and $\gamma_1, \gamma_2, \ldots , \gamma_{n-1} \in \Gamma$. If $d(x\gamma y) = d(x)\gamma y+x\gamma d(y)$
for all $x, y \in \M$, we simply say that $d$ is a multiplicative derivation of $\M$. Ferreira and Ferreira in \cite{bruth1} and Ferreira in \cite{bruth2} obtained the followings theorems.

\begin{theorem}\label{passbru} Let $\mathfrak{M}$ be a $\Gamma$-ring containing a family $\{e_{\alpha}|\alpha\in \Lambda\}$ of $\gamma _{\alpha}$-idempotents which satisfies:
\begin{enumerate}
\item[\it (i)] For each $\alpha \in \Lambda$ there are two $\mathfrak{M}$-additive maps $f_{\alpha} \colon \Gamma \times\mathfrak{M}\rightarrow \mathfrak{M}$, \linebreak $f'_{\alpha} \colon \mathfrak{M}\times \Gamma \rightarrow \mathfrak{M}$ satisfying $f_{\alpha} (\gamma_{\alpha}, a)=a-e_{\alpha}\gamma_{\alpha}a$, $f'_{\alpha} (a,\gamma_{\alpha})=a-a\gamma_{\alpha}e_{\alpha}$, for all $a\in \mathfrak{M}$, such that if we denote $f_{\alpha}\beta a=f_{\alpha} (\beta, a)$, $a \beta f_{\alpha}=f'_{\alpha} (a,\beta)$, $1_{\alpha}\beta a=e_{\alpha}\beta a+f_{\alpha}\beta a$, $a\beta 1_{\alpha}=a\beta e_{\alpha}+a\beta f_{\alpha}$ (allowing us to write \linebreak $1_{\alpha}=e_{\alpha}+f_{\alpha}$), then $(a\beta f_{\alpha})\gamma b=a \beta (f_{\alpha}\gamma b)$ for all $a,b\in \mathfrak{M}$ and $\beta, \gamma \in \Gamma$;
\item[\it (ii)] If $x\in \mathfrak{M}$ is such that $x\Gamma \mathfrak{M}=0,$ then $x = 0$;
\item[\it (iii)] If $x\in \mathfrak{M}$ is such that $e_{\alpha}\Gamma\mathfrak{M}\Gamma x = 0$ for all $\alpha\in \Lambda,$ then $x = 0$ (and hence $\mathfrak{M}\Gamma x=0$ implies $x = 0$);
\item[\it (iv)] For each $\alpha \in \Lambda$ and $x\in \mathfrak{M},$ if $(e_{\alpha} \gamma _{\alpha} x \gamma _{\alpha} e_{\alpha})\Gamma\mathfrak{M}\Gamma(1_{\alpha}-e_{\alpha})=0$ then $e_{\alpha} \gamma _{\alpha} x \gamma _{\alpha} e_{\alpha} = 0.$
\end{enumerate}
Then any multiplicative isomorphism ($\varphi$, $\phi$) of $\M$ onto an arbitrary gamma ring or multiplicative derivation of $\M$ is additive.
\end{theorem}

Inspired by the above-mentioned results in the present article we consider the same Wang’s problems in the context of $\Gamma$-rings. We are planning to extend Theorem \ref{passbru} to an arbitrary $n$-multiplicative isomorphism and $n$-multiplicative derivations in Section 3, for that we borrow the Wang's techniques in order to results presented are generalizations of the results of Wang to the class of $\Gamma$-rings. In
the Section 4, we give the applications of our main result.

\section{Main result}

We shall prove the following main result of this article.
\begin{theorem}\label{main} Let $k$ be a positive integer. Let $\M$ be a $\Gamma$-ring satisfying the
conditions $(i)-(iv)$ of Theorem \ref{passbru}. Suppose that $f \colon  \M \times \Gamma \times \M \rightarrow \M$ is a mapping such that:
\begin{enumerate}
\item [(v)] $f(x, \gamma, 0) = f(0, \gamma, x) = 0$ for all $x \in \M$ and $ \gamma \in \Gamma$,
\item [(vi)] $u_1\gamma_1u_2 \cdots  \gamma_{k-1}u_k\gamma_k f(x, \gamma, y) = f(u_1\gamma_1u_2 \cdots  \gamma_{k-1}u_k\gamma_k x, \gamma, u_1\gamma_1u_2 \cdots  \gamma_{k-1}u_k\gamma_k y)$,
\item [(vii)] $f(x, \gamma, y)\gamma_1 u_1 \gamma_2 u_2 \cdots  \gamma_k u_k = f(x\gamma_1 u_1\gamma_2 u_2 \cdots  \gamma_k u_k, \gamma, y\gamma_1 u_1\gamma_2 u_2 \cdots  \gamma_k u_k)$,
\end{enumerate}
for all $x, y, u_1, u_2, \ldots  , u_k \in \M$ and $\gamma, \gamma_1, \ldots  , \gamma_k \in \Gamma$. Then $f = 0$.
\end{theorem}

Henceforth we always assume that $\M$ satisfies the conditions $(i)-(iv)$ of Theorem \ref{passbru}. The proof will be organized in a series of claims. We begin with the following claim.

\begin{claim}\label{claim1}
$u\beta f(x, \gamma, y) = f(u \beta x, \gamma, u \beta y)$ and $f(x, \gamma, y)\beta u = f(x\beta u, \gamma, y\beta u)$ for all $x, y, u \in \M$ and $\beta, \gamma \in \Gamma$.
\end{claim}
\begin{proof}
For any $x, y, u, u_1, u_2, \ldots  , u_k \in \M$ and $\beta, \gamma, \gamma_1, \gamma_2, \ldots  , \gamma_k \in \Gamma$, we have 
$$
\begin{aligned}
f(x, \gamma, y)\beta u\gamma_1u_1\gamma_2u_2 \cdots  \gamma_k u_k &=f(x, \gamma, y)\beta(u\gamma_1 u_1)\gamma_2 u_2 \cdots  \gamma_k u_k
\\&= f(x\beta u\gamma_1 u_1 \gamma_2 u_2 \cdots   \gamma_{k} u_k, \gamma, y\beta u\gamma_1 u_1 \gamma_2 u_2 \cdots  \gamma_k u_k)
\\&= f(x\beta u, \gamma, y\beta u)\gamma_1 u_1 \gamma_2 u_2 \cdots  \gamma_k u_k,
\end{aligned}
$$
by condition $(vii)$. Hence, $(f(x, \gamma, y)\beta u - f(x\beta u, \gamma, y\beta u)) \gamma_1 u_1 \gamma_2 u_2 ... \gamma_k u_k = 0.$
This implies
$$(\cdots((f(x, \gamma, y)\beta u - f(x\beta u, \gamma, y\beta u)) \Gamma \M) \Gamma \M \cdots )\Gamma \M = 0.$$
From the condition $(ii)$ of the Theorem \ref{passbru}, we have $f(x, \gamma, y)\beta u = f(x\beta u, \gamma, y\beta u).$
Similarly, we prove $u \beta f(x, \gamma, y) = f(u \beta x, \gamma, u \beta y).$
\end{proof}

The conditions of the Theorem \ref{passbru}, ensure that $\Lambda \neq \emptyset$. Thus, let $e_1 \in \left\{e_\alpha ~|~ \alpha \in \Lambda \right\}$ be a nontrivial $\gamma_1$-idempotent.

\begin{claim}\label{claim2}
$f(x_{ii}, \gamma, x_{jk}) = 0 = f(x_{jk}, \gamma, x_{ii})$ for $j \neq k$.
\end{claim}

\begin{proof}
First assume that $i = j$. For arbitrary $a_{rs} \in \M_{rs} (r, s = 1, 2)$ we have
$f(x_{ii}, \gamma, x_{jk})\gamma_1e_l\beta a_{rs} = f(x_{ii}\gamma_1e_l\beta a_{rs}, \gamma, x_{jk}\gamma_1e_l\beta a_{rs}) = 0 \ (l = 1, 2)$ which
yields in $f(x_{ii}, \gamma, x_{jk})\gamma_1e_l\beta a = 0 \ (l = 1, 2)$. This implies $f(x_{ii}, \gamma, x_{jk})\gamma_1 1_1 \Gamma \M = 0$ which results in
$$f(x_{ii}, \gamma, x_{jk})\Gamma \M = 0.$$
By condition $(ii)$ of the Theorem \ref{passbru}, we see that $f(x_{ii}, \gamma, x_{jk}) = 0$. Now, if $i \neq j$ then we have $i = k$. Hence, for arbitrary $a_{rs} \in \M_{rs} \ (r, s = 1, 2)$ we have $a_{rs}\beta e_l \gamma_1 f(x_{ii}, \gamma, x_{jk}) = f(a_{rs}\beta e_l \gamma_1 x_{ii}, \gamma, a_{rs}\beta e_l\gamma_1 x_{jk}) = 0 \ (l = 1, 2)$ which
yields in $a\beta e_l \gamma_1 f(x_{ii}, \gamma, x_{jk}) = 0 \ (l = 1, 2)$. This implies $\M \Gamma 1_1 \gamma_1 f(x_{ii}, \gamma, x_{jk}) =
0$ which results in $\M \Gamma f(x_{ii}, \gamma, x_{jk}) = 0$. By condition $(iii)$ of the Theorem \ref{passbru},
we see that $f(x_{ii}, \gamma, x_{jk}) = 0$. Similarly, we prove $f(x_{jk}, \gamma, x_{ii}) = 0.$
\end{proof}

\begin{claim}\label{claim3}
$f(x_{12}, \gamma, u_{12}) = 0.$
\end{claim}
\begin{proof}
For an arbitrary $a_{rs} \in \M_{rs} \ (r, s = 1, 2)$, we have $f(x_{12}, \gamma, u_{12})\gamma_1e_1\beta a_{rs} = f(x_{12}\gamma_1e_1\beta a_{rs}, \gamma, u_{12}\gamma_1e_1\beta a_{rs}) = f(0, \gamma, 0) = 0$, by Claim \ref{claim1}. It follows that
\begin{eqnarray}
f(x_{12}, \gamma, u_{12})\gamma_1e_1\beta a = 0,
\end{eqnarray}
where $a = a_{11} + a_{12} + a_{21} + a_{22}$. Now, from Claim \ref{claim2}, we have 
$$ \begin{aligned}f(x_{12}, \gamma, u_{12})&\gamma_1e_2\beta a_{1l} = f(x_{12}\gamma_1e_2\beta a_{1l}, \gamma, u_{12}\gamma_1e_2\beta a_{1l}) \\&= f(x_{12}\gamma_1(e_2\beta a_{1l} + u_{12}\gamma_1e_2\beta a_{1l}), \gamma, e_1\gamma_1(e_2\beta a_{1l} + u_{12}\gamma_1e_2\beta a_{1l}))\\&= f(x_{12}, \gamma, e_1)(\gamma_1(e_2\beta a_{1l} + u_{12}\gamma_1e_2\beta a_{1l}))= 0 \ (l = 1, 2)\end{aligned}$$ 
and 
$$ \begin{aligned}f(x_{12}, \gamma, u_{12})&\gamma_1e_2\beta a_{2l} = f(x_{12}\gamma_1e_2\beta a_{2l}, \gamma, u_{12}\gamma_1e_2\beta a_{2l}) \\&=f(x_{12}\gamma_1(e_2\beta a_{2l} + u_{12}\gamma_1e_2\beta a_{2l}), \gamma, e_1\gamma_1(e_2\beta a_{2l} + u_{12}\gamma_1e_2\beta a_{2l})) \\&= f(x_{12}, \gamma, e_1)(\gamma_1(e_2\beta a_{2l} + u_{12}\gamma_1e_2\beta a_{2l}))= 0 \ (l = 1, 2).
\end{aligned}
$$ It follows that
\begin{eqnarray}
f(x_{12}, \gamma, u_{12})\gamma_1e_2\beta a = 0,
\end{eqnarray}
where $a = a_{11} + a_{12} + a_{21} + a_{22}$. From $(1)$ and $(2)$, we obtain \linebreak $f(x_{12}, \gamma, u_{12})\gamma_1 1_1\beta a =0$ which results in
$$f(x_{12}, \gamma, u_{12})\Gamma \M = 0.$$
This implies $f(x_{12}, \gamma, u_{12}) = 0$, by condition $(ii)$ of the Theorem \ref{passbru}.
\end{proof}

\begin{claim}\label{claim4}
$f(x_{11}, \gamma, u_{11}) = 0.$
\end{claim}
\begin{proof}
First of all, note that $e_1f(x_{11}, \gamma, u_{11})e_1 = \gamma_1f(x_{11}, \gamma, u_{11})\gamma_1$. On the other hand, for arbitrary element $a \in \M$ we have $f(x_{11}, \gamma, u_{11})\gamma_1e_l\alpha a \beta e_2 = f(x_{11}\gamma_1 e_l\alpha a \beta e_2, \gamma, u_{11} \gamma_1 e_l \alpha a \beta e_2) = 0 \ (l = 1, 2)$, by Claim \ref{claim3}. It follows that
$f(x_{11}, \gamma, u_{11})\gamma_1 1_1 \alpha a \beta e_2 = 0$ which implies
$$(e_1 f(x_{11}, \gamma, u_{11})e_1)\Gamma \M \Gamma (1_1 - e_1) = 0.$$
Thus, $f(x_{11}, \gamma, u_{11}) = 0$, by condition $(iv)$ of the Theorem \ref{passbru}.
\end{proof}

\begin{claim}\label{claim4}
$f(e_1\Gamma \M, e_1 \Gamma \M) = 0.$
\end{claim}

\begin{proof}
Let $x, y \in \M$ and $\lambda, \mu \in \Gamma$ be arbitrary elements and let us write
$x = x_{11} + x_{12} + x_{21} + x_{22}$ and $y = y_{11} + y_{12} + y_{21} + y_{22}$. It follows that $e_1\lambda x =
e_1\lambda x_{11}+e_1\lambda x_{12}+e_1 \lambda x_{21}+ e_1\lambda x_{22}$ and $e_1\mu y = e_1 \mu y_{11}+e_1 \mu y_{12}+e_1 \mu y_{21}+e_1 \mu y_{22}$.
Hence, for an arbitrary element $a_{ij} \in \M_{ij} (i, j = 1, 2)$ and $\beta \in \Gamma$, by properties
of Peirce decomposition of $\M$ and making use of the Claims \ref{claim1}, \ref{claim3} and
\ref{claim4}, we can see that 
$$
\begin{aligned}
f(e_1\lambda x, &\gamma, e_1 \mu y)\gamma_1e_1 \beta a_{ij} = f((e_1 \lambda x_{11} + e_1 \lambda x_{12} + e_1 \lambda x_{21} +
e_1 \lambda x_{22})\gamma_1e_1\beta a_{ij} , \\& \gamma,(e_1 \mu y_{11} +e_1 \mu y_{12} +e_1 \mu y_{21} +e_1 \mu y_{22}) \gamma_1e_1 \beta a_{ij})
\\&= f((e_1 \lambda x_{11} + e_1\lambda x_{21})\gamma_1 e_1 \beta a_{ij} , \gamma,(e_1 \mu y_{11}+e_1 \mu y_{21})\gamma_1e_1 \beta a_{ij})
\\&= 0
\end{aligned}
$$
and
$$
\begin{aligned}
 f(e_1 \lambda x, &\gamma, e_1\mu y)\gamma_1e_2 \beta a_{ij}
= f((e_1 \lambda x_{11}+e_1 \lambda x_{12}+ e_1 \lambda x_{21}+ e_1 \lambda x_{22})\gamma_1 e_2 \beta a_{ij}, \\& \gamma, (e_1 \mu y_{11}+e_1\mu y_{12}+e_1\mu y_{21}+
e_1 \mu y_{22})\gamma_1e_2\beta a_{ij}) \\&= f((e_1\lambda x_{12}+e_1\lambda x_{22})\gamma_1e_2\beta a_{ij} , \gamma, (e_1\mu y_{12}+e_1\mu y_{22})\gamma_1e_2\beta a_{ij})
\\&= 0. 
\end{aligned}
$$

It follows that $f(e_1\lambda x, \gamma , e_1\mu y)\gamma_1 1_1 \beta a = 0$, where $a = a_{11} +a_{12} +a_{21} +a_{22}.$
Thus $f(e_1\lambda x, \gamma, e_1 \mu y)\Gamma \M = 0$ which implies $f(e_1\lambda x, \gamma, e_1\mu y) = 0$, by condition $(ii)$ of the Theorem \ref{passbru}.
\end{proof}

\textbf{Proof of Theorem \ref{main}}
\begin{proof} 
Suppose $x, y \in \M$. For arbitrary elements $\alpha \in \Lambda$, $r \in \M$ and $\beta, \gamma \in \Gamma$
we have $e_\alpha \beta r\gamma f(x, \gamma, y) = f(e_\alpha \beta r \gamma x, \gamma, e_\alpha \beta r \gamma y) = 0$. Hence, for all $\alpha \in \Lambda$ we
have $e_\alpha \Gamma \M \Gamma f(x, \gamma, y) = 0$ which implies $f(x, \gamma, y) = 0$, by condition $(iii)$ of the
Theorem \ref{passbru}.
The Theorem is proved.
\end{proof}

\section{Applications of Theorem \ref{main}}

We will give some applications of our main result.

\begin{theorem}\label{teo1}
Let $\M$ be a $\Gamma$-ring satisfying the conditions $(i)-(iv)$ of Theorem
\ref{passbru}. Then any $n$-multiplicative isomorphism $\left(\varphi, \phi\right)$ of $\M$ onto an arbitrary
gamma ring is additive.
\end{theorem}

\begin{proof}
Since $\varphi$ is onto, $\varphi(x) = 0$ for some $x \in \M$. Hence $\varphi(0) = \varphi(0\gamma_1 \cdots  0\gamma_{n-1}x) =
\varphi(0)\phi(\gamma_1)\cdots\varphi(0)\phi(\gamma_{n-1})\varphi(x) = \varphi(0)\phi(\gamma_1)\cdots\varphi(0)\phi(\gamma_{n-1})0 = 0$ resulting in
$\varphi^{-1}(0) = 0$. Now, for any $x, y \in \M$ and $\gamma \in \Gamma$ we set
$$f(x, \gamma, y) = \varphi^{-1}(\varphi(x + y) - \varphi(x) - \varphi(y)).$$
Let us show that $f(x, \gamma, 0) = f(0, \gamma, x) = 0$ for all $x \in \M$. It is easy to
check that $(\varphi^{-1}, \phi^{-1})$ is also a $n$-multiplicative isomorphism. Thus, for any $u_1,\ldots, u_{n-1} \in \M$ and $\gamma_1, \ldots  , \gamma_{n-1} \in \Gamma$, we have
$$
\begin{aligned}
&f(x, \gamma, y)\gamma_1u_1 \cdots  \gamma_{n-1}u_{n-1}
\\& = \varphi^{-1}(\varphi(x + y) - \varphi(x) - \varphi(y))\phi^{-1}(\phi(\gamma_1))\varphi^{-1}(\varphi(u_1))
\\& \ \ \ \ \ \ \ \cdots  \phi^{-1}(\phi(\gamma_{n-1}))\varphi^{-1}(\varphi(u_{n-1}))
\\& = \varphi^{-1}((\varphi(x + y) - \varphi(x) - \varphi(y))
\phi(\gamma_1)\varphi(u_1)\cdots  \phi(\gamma_{n-1})\varphi(u_{n-1}))
\\& = f(x\gamma_1u_1 \cdots  \gamma_{n-1}u_{n-1}, \gamma, y\gamma_1u_1 \cdots  \gamma_{n-1}u_{n-1}).
\end{aligned}
$$
Similarly, we prove
$$u_1\gamma_1 \cdots  u_{n-1}\gamma{n-1}f(x, \gamma, y) = f(u_1\gamma_1 \cdots  u_{n-1}\gamma_{n-1}x, \gamma, u_1\gamma_1 \cdots  u_{n-1}\gamma_{n-1}y).$$
By Theorem \ref{main}, we have $f(x, \gamma, y) = 0$, for all $x, y \in \M$ and $\gamma \in \Gamma$, which
implies $\varphi(x + y) = \varphi(x) + \varphi(y)$ for all $x, y \in \M$.
\end{proof}

\begin{corollary}
Let $\M$ be a prime $\Gamma$-ring containing a $\gamma_1$-idempotent and a
$\gamma_1$-unity element, where $\gamma_1 \in \Gamma$. Then any $n$-multiplicative isomorphism $\left(\varphi, \phi\right)$ of $\M$ onto an arbitrary gamma ring is additive.
\end{corollary}

As a consequence of Theorem \ref{teo1}, we have the following Corollary that has been proven by Ferreira and Ferreira in \cite{bruth1}.

\begin{corollary}$\left[\cite{bruth1}, \mbox{Theorem} \ 2.1\right]$
Let $\mathfrak{M}$ be a $\Gamma$-ring containing a family $\{e_{\alpha}|\alpha \in \Lambda\}$ of nontrivial $\gamma _{\alpha}$-idempotents which satisfies:
\begin{enumerate}
\item[\it (i)] If $x\in \mathfrak{M}$ is such that $x\Gamma \mathfrak{M}=0,$ then $x = 0$;
\item[\it (ii)] If $x\in \mathfrak{M}$ is such that $e_{\alpha}\Gamma\mathfrak{M}\Gamma x = 0$ for all $\alpha\in \Lambda,$ then $x = 0$ (and hence $\mathfrak{M}\Gamma x=0$ implies $x = 0$);
\item[\it (iii)] For each $\alpha \in \Lambda$ and $x\in \mathfrak{M},$ if $(e_{\alpha} \gamma _{\alpha} x \gamma _{\alpha} e_{\alpha})\Gamma\mathfrak{M}\Gamma(1_{\alpha}-e_{\alpha})=0$ then $e_{\alpha} \gamma _{\alpha} x \gamma _{\alpha} e_{\alpha} = 0.$
\end{enumerate}
Then any multiplicative isomorphism  $(\varphi, \phi)$ of $\mathfrak{M}$ onto an arbitrary gamma ring is additive.
\end{corollary}

As an Corollary of our main result, we obtain a very short proof of a result
of Ferreira \cite{bruth3} as follows.

\begin{corollary}$\left[\cite{bruth3}, \mbox{Theorem} \ 2.1\right]$
Let $\M, \M'$, $\Gamma$ and $\Gamma'$ be additive groups such that $\M$ is a $\Gamma$-ring and $\M'$ is a $\Gamma'$-ring. Suppose that $\M$ contain a family
$\left\{e_\alpha ~|~ \alpha \in \Lambda \right\}$ of $\gamma_\alpha$-idempotents which satisfies:

\begin{enumerate}

\item [(i)] For each $\alpha \in \Lambda$ there are two $\M$-additive maps $f_\alpha\colon  \Gamma \times \M \rightarrow \M$, $f'_{\alpha}\colon  \M \times \Gamma \rightarrow \M$ satisfying $f_{\alpha}(\gamma_{\alpha}, a) = a - e_{\alpha}\gamma_{\alpha}a$, $f'_{\alpha}(a, \gamma_{\alpha}) = a - a\gamma_{\alpha}e_{\alpha}$, for all $a \in \M$, such that if we denote $f_{\alpha}\beta a = f_{\alpha}(\beta, a)$, $a\beta f_{\alpha} =f'_{\alpha}(a, \beta)$, $1_{\alpha}\beta a = e_{\alpha}\beta a+f_{\alpha} \beta a$, $a\beta 1_{\alpha} = a\beta e_{\alpha} + a\beta f_{\alpha}$ (allowing us to write $1_{\alpha} = e_{\alpha} + f_{\alpha})$, then $(a \beta f_{\alpha}) \gamma b = a \beta (f_{\alpha} \gamma b)$ for all $\beta, \gamma \in \Gamma$ and $a, b \in \M$;
\item [(ii)] If $x \in \M$ is such that $x\Gamma \M = 0$, then $x = 0$;
\item [(iii)] If $x \in \M$ is such that $e_{\alpha}\Gamma \M \Gamma x = 0$ for all $\alpha \in \Lambda$, then $x = 0$ (and hence $\M \Gamma x = 0$ implies $x = 0$);
\item [(iv)] For each $\alpha \in \Lambda$ and $x \in \M$, if $(e_{\alpha}\gamma_{\alpha}x \gamma_{\alpha} e_{\alpha})\Gamma \M \Gamma (1_{\alpha} - e_{\alpha}) =0$ then $e_{\alpha}\gamma_{\alpha}x \gamma_{\alpha}e_{\alpha} = 0$.

\end{enumerate}

Then every surjective elementary map $(M, M^{*})$ of $\M \times \M'$ is additive.
\end{corollary}

\begin{proof}
In view of Ferreira [\cite{bruth3}, Lemma 2.2], we know that both $\M$
and $\M^{*}$ are bijective maps. For $x, y \in \M$ and $\gamma \in \Gamma$, we set $$f(x, \gamma, y) = M^{-1}(M(x + y) - M(x) - M(y)).$$
Also by Ferreira [\cite{bruth3}, Lemma 2.1], we have $M(0) = 0'$ and so
$M^{-1}(0') = 0$. Thus $f(x, \gamma, 0) = 0 = f(0, \gamma, x)$ for all $x \in \M$. For any $x, y, u, v \in
\M$ and $\beta, \delta \in \Gamma$ by assumption, we have
$$
\begin{aligned}
&M(f(x\beta u \delta v, \gamma, y\beta u \delta v))
\\&= M((x + y)\beta u \delta v)- M(x\beta u\delta v) - M(y\beta u\delta v)
\\& = M((x + y)\beta M^{*}{M^{*}}^{-1}(u)\delta v)-M(x\beta M^{*}{M^{*}}^{-1}(u)\delta v) - M(y\beta M^{*}{M^{*}}^{-1}(u) \delta v)
\\& = M(x + y)\phi(\beta){M^{*}}^{-1}(u)\phi(\delta)M(v)-M(x)\phi(\beta){M^{*}}^{-1}(u)\phi(\delta)M(v) 
\\&- M(y)\phi(\beta){M^{*}}^{-1}(u)\phi(\delta)M(v)
\\&=(M(x + y) - M(x) - M(y))\phi(\beta){M^{*}}^{-1}(u)\phi(\delta)M(v)
\\&= M(f(x, \gamma, y))\phi(\beta){M^{*}}^{-1}(u)\phi(\delta)M(v)
\\&= M(f(x, \gamma, y)\beta u\delta v).
\end{aligned}
$$
Hence $f(x\beta u \delta v, \gamma, y\beta u\delta ) = f(x, \gamma, y)\beta u \delta v$. 

Similarly, we prove $f(u \beta v \delta x, \gamma, u \beta v \delta y) =
u\beta v \delta f(x, \gamma, y)$. By Theorem 2.2, we have $f = 0$ and so $M$ is additive. Since
$({M^{*}}^{-1}, M^{-1})$ is also an elementary map of $\M \times \M'$
[\cite{bruth3}, Lemma 2.3], we also get
that ${M^{*}}^{-1}$ is additive and so $M^{*}$ is additive. This complete the proof.
\end{proof}

We now discuss the additivity of $n$-multiplicative derivation.

\begin{corollary}\label{deriv}
Let $\M$ and $\Gamma$ be additive groups such that $\M$ is a $\Gamma$-ring. Suppose that $\M$ contain a family $\left\{e_{\alpha} ~|~ \alpha \in \Lambda \right\}$ of $\gamma_{\alpha}$-idempotents which satisfies:

\begin{enumerate}

\item [(i)] For each $\alpha \in \Lambda$ there are two $\M$-additive maps $f_\alpha\colon  \Gamma \times \M \rightarrow \M$, $f'_{\alpha}\colon  \M \times \Gamma \rightarrow \M$ satisfying $f_{\alpha}(\gamma_{\alpha}, a) = a - e_{\alpha}\gamma_{\alpha}a$, $f'_{\alpha}(a, \gamma_{\alpha}) = a - a\gamma_{\alpha}e_{\alpha}$, for all $a \in \M$, such that if we denote $f_{\alpha}\beta a = f_{\alpha}(\beta, a)$, $a\beta f_{\alpha} =f'_{\alpha}(a, \beta)$, $1_{\alpha}\beta a = e_{\alpha}\beta a+f_{\alpha} \beta a$, $a\beta 1_{\alpha} = a\beta e_{\alpha} + a\beta f_{\alpha}$ (allowing us to write $1_{\alpha} = e_{\alpha} + f_{\alpha})$, then $(a \beta f_{\alpha}) \gamma b = a \beta (f_{\alpha} \gamma b)$ for all $\beta, \gamma \in \Gamma$ and $a, b \in \M$;
\item [(ii)] If $x \in \M$ is such that $x\Gamma \M = 0$, then $x = 0$;
\item [(iii)] If $x \in \M$ is such that $e_{\alpha}\Gamma \M \Gamma x = 0$ for all $\alpha \in \Lambda$, then $x = 0$ (and hence $\M \Gamma x = 0$ implies $x = 0$);
\item [(iv)] For each $\alpha \in \Lambda$ and $x \in \M$, if $(e_{\alpha}\gamma_{\alpha}x \gamma_{\alpha} e_{\alpha})\Gamma \M \Gamma (1_{\alpha} - e_{\alpha}) =0$ then $e_{\alpha}\gamma_{\alpha}x \gamma_{\alpha}e_{\alpha} = 0$.

\end{enumerate}
Then any $n$-multiplicative derivation $d$ of $\M$ is additive.
\end{corollary}

\begin{proof}
By our assummption, we easily see that $d(0) = 0$. For any $x, y \in \M$ and
$\gamma \in \Gamma$, we set $f(x, \gamma, y) = d(x+y)-d(x)-d(y)$. Thus $f(x, \gamma, 0) = 0 = f(0, \gamma, x)$ for all $x \in \M$. Futhermore, it is easy to check that the conditions $(ii)$ and $(iii)$
in Theorem \ref{main} are met. So the result follows from Theorem \ref{main}.
\end{proof}

\begin{corollary}
Let $\M$ be a prime $\Gamma$-ring containing a $\gamma_1$-idempotent and a
$\gamma_1$-unity element, where $\gamma_1 \in \Gamma$. Then any $n$-multiplicative derivation $d$ of $\M$ is additive.
\end{corollary}

And one last corollary we have proved Ferreira's result in \cite{bruth2}.

\begin{corollary}$\left[\cite{bruth2}, \mbox{Theorem} \ 2.1\right]$
Let $\mathfrak{M}$ be a $\Gamma$-ring containing a family $\{e_{\alpha}|\alpha\in \Lambda\}$ of $\gamma _{\alpha}$-idempotents which satisfies:
\begin{enumerate}
\item[\it (i)] For each $\alpha \in \Lambda$ there are two $\mathfrak{M}$-additive maps $f_{\alpha} \colon \Gamma \times\mathfrak{M}\rightarrow \mathfrak{M}$, \linebreak $f'_{\alpha} \colon \mathfrak{M}\times \Gamma \rightarrow \mathfrak{M}$ satisfying $f_{\alpha} (\gamma_{\alpha}, a)=a-e_{\alpha}\gamma_{\alpha}a$, $f'_{\alpha} (a,\gamma_{\alpha})=a-a\gamma_{\alpha}e_{\alpha}$, for all $a\in \mathfrak{M}$, such that if we denote $f_{\alpha}\beta a=f_{\alpha} (\beta, a)$, $a \beta f_{\alpha}=f'_{\alpha} (a,\beta)$, $1_{\alpha}\beta a=e_{\alpha}\beta a+f_{\alpha}\beta a$, $a\beta 1_{\alpha}=a\beta e_{\alpha}+a\beta f_{\alpha}$ (allowing us to write \linebreak $1_{\alpha}=e_{\alpha}+f_{\alpha}$), then $(a\beta f_{\alpha})\gamma b=a \beta (f_{\alpha}\gamma b)$ for all $a,b\in \mathfrak{M}$ and $\beta, \gamma \in \Gamma$;
\item[\it (ii)] If $x\in \mathfrak{M}$ is such that $x\Gamma \mathfrak{M}=0,$ then $x = 0$;
\item[\it (iii)] If $x\in \mathfrak{M}$ is such that $e_{\alpha}\Gamma\mathfrak{M}\Gamma x = 0$ for all $\alpha\in \Lambda,$ then $x = 0$ (and hence $\mathfrak{M}\Gamma x=0$ implies $x = 0$);
\item[\it (iv)] For each $\alpha \in \Lambda$ and $x\in \mathfrak{M},$ if $(e_{\alpha} \gamma _{\alpha} x \gamma _{\alpha} e_{\alpha})\Gamma\mathfrak{M}\Gamma(1_{\alpha}-e_{\alpha})=0$ then $e_{\alpha} \gamma _{\alpha} x \gamma _{\alpha} e_{\alpha} = 0.$
\end{enumerate}
Then any multiplicative derivation $d$ of $\mathfrak{M}$ is additive.
\end{corollary}

\end{document}